\documentclass[12pt]{article}
\usepackage[english]{babel}
\usepackage[T1]{fontenc}
\usepackage{amsmath}
\usepackage{anysize}
\usepackage{amsfonts}
\usepackage{amssymb}
\usepackage{graphicx}
\usepackage{float}
\usepackage{color}
\usepackage{xr}
\usepackage{makeidx}
\usepackage[all]{xy}
\usepackage{hyperref}
\numberwithin{equation}{section}
\newenvironment{proof}{\paragraph{Proof:}}{\hfill$\square$}
\newcommand{\til}{\widetilde}
\newcommand{\vol}{\mathrm{Vol}}

\newcommand{\lied}{\mathcal{L}}

\newcommand{\dps}{\displaystyle}

\numberwithin{equation}{section}

\newcommand{\tl}{\textless}
\newcommand{\tgt}{\textgreater}
\newtheorem{teo}{\textit{Theorem}}
\newtheorem{pro}{\textit{Proposition}}[section]
\newtheorem{dfn}{\textit{Definition}}[section]

\title{Invariant Solutions of the Yamabe equation on the Koiso-Cao Soliton}

\author{J. TORRES OROZCO}

\date{CIMAT, 2016}

\begin{document}

\maketitle

\begin{abstract}
We consider the non-trivial Ricci soliton on $\mathbb{CP}^2\#\overline{\mathbb{CP}^2}$ constructed by Koiso and Cao. It is a K\"ahler metric invariant by the $U(2)$ action on $\mathbb{CP}^2\#\overline{\mathbb{CP}^2}$. We study its Yamabe equation and prove it has exactly one $U(2)-$invariant solution up to homothecies. 
\end{abstract}


\section{Introduction}

Let $(M, g)$ be a closed Riemannian manifold of dimension $n$, $n\geq 3$. Let $\til{g}=f^{p-2}\cdot g$, where $p=\frac{2n}{n-2}$, and $f$ is a smooth positive function on $M$. Let $S_g$ and $S_{\til g}$ be the scalar curvatures of $g$ and $\til g$ respectively. If one compute $S_{\til g}$ in terms of $g$ and $f$:
\begin{equation*}
S_{\til g}=f^{1-p}\left(-\frac{4(n-1)}{n-2}\triangle_g f+S_g f\right).
\end{equation*}

\noindent{Thus $\til{g}$ has constant scalar curvature $\lambda$
if and only if $f$ satisfies the\!\!\textit{\textbf{{
Yamabe equation:}}}}
\begin{equation}
\label{yamabequa} -a_n\triangle_g f+ S_g f=\lambda f^{p-1}
\end{equation}

\noindent where $a_n=\frac{4(n-1)}{n-2}$.
\newline

On the space of Riemannian metrics on $M$ consider the Hilbert--Einstein functional $\mathcal{S}$:
\begin{equation*}
g\mapsto \mathcal{S}(g)=\frac{\int_M S_g dv_g}{\vol(M, g)^{\frac{n-2}{n}}}
\end{equation*}

\noindent where $dv_g$ is the volume element of $g$, and $\vol(M, g)=\int_M dv_g$ is the volume of $M$.
\newline

Denote by $[g]$ the conformal class of $g$, $[g]=\{f^{p-2}\cdot g| f\in C^{\infty}(M), f>0\}$. The volume element of $f^{p-2}\cdot g$ is $dv_{f^{p-2}\cdot g}=f^{\frac{n}{2}(p-2)}\ dv_g=f^{p}\ dv_g$. Then, if we restrict the Hilbert--Einstein functional to a conformal
class, we obtain:
\begin{eqnarray*}
\mathcal{S}(f^{p-2}\cdot g)&=&\frac{\int_M S_{\til g}\ dv_{\til g}}{\vol(M, \til g )^{\frac{n-2}{n}}}\\
&=&\frac{\int_M f^{1-p}(-a_n\triangle_g f+S_g f)f^p\ dv_g}{(\int_M f^p\ dv_g)^{\frac{2}{p}}}\\
&=&\frac{(\int_M a_n |\nabla f|^2+S_g f^2)\ dv_g}{||f||_p^2}.
\end{eqnarray*}

\begin{dfn} The \textbf{\textit{Yamabe functional}} of a Riemannian manifold $(M, g)$ of dimension $n\geq 3$ is defined as:
\begin{equation}
\label{yamabefun} \mathcal{Y}_g(f)=\frac{(\int_M a_n |\nabla
f|^2+S_g f^2)\ dv_g}{||f||_p^2}.
\end{equation}
\end{dfn}

This functional was first considered by Yamabe in \cite{Yamabe}. 
\newline

The Euler-Lagrange equation of $\mathcal{Y}_g$ is exactly the Yamabe equation, and therefore critical points of $\mathcal{S}|_{[g]}$ are metrics of constant scalar curvature in $[g]$. It is easy to see that $\mathcal{Y}_g$ is bounded below, then one may define:

\begin{dfn}
The \textbf{Yamabe constant} of a conformal class $[g]$ is:

$$Y(M, [g])=\inf_{h\in [g]}\mathcal{S}(h).$$

If $h\in [g]$ satisfies $\mathcal{S}(h)=Y(M, [g])$, i.e.,
$h$ realizes the infimum, then $h$ has constant scalar curvature.
These minimizing metrics are called \textbf{Yamabe metrics}.
\end{dfn}

A fundamental theorem of Yamabe-Trudinger-Aubin-Schoen (\cite{Yamabe}, \cite{Trudinger}, \cite{Aubin}, \cite{Schoen}) asserts that the infimum is always achieved. Then, there exists a constant scalar curvature metric in every conformal class. 
\newline

If $Y(M, [g])\leq 0$, there is exactly one metric in $[g]$ of constant scalar curvature up to homothecies. Also, if $g$ is Einstein and it  is not the round metric on the sphere, M. Obata proved that one also has uniqueness \cite{Obata}. But in general there are multiple solutions (see for instance \cite{Brendle}, \cite{HenryPetean}, \cite{LimaPicc}, \cite{Petean1}, \cite{Pollack}). It is an important problem to understand if there are other general situations where uniqueness holds. For this purpose we will focus on Ricci solitons.
\medskip

\begin{dfn}[Ricci Soliton]
Let $(M, g)$ be a Riemannian manifold of dimension $n$ such that
\begin{equation}
-2Ric(g)=\lied_X g+2\mu g \label{solitricci}
\end{equation}
holds for some constant $\mu$ and some complete vector field X on
$M$. We say $g$ is a Ricci soliton. Here $Ric(g)$ is the
Ricci curvature of $M$ and $\dps{\lied_X g}$ is the Lie derivative
of the metric in the direction of $X$.
\end{dfn}

If $\mu$ is negative, zero or positive $g$ is said to be a \textbf{shrinking, steady,} or \textbf{expanding} soliton. If $X\equiv0$, then $Ric(g)=-\mu g$, and $g$ is Einstein. Then a Ricci soliton may be regarded as a generalization of an Einstein metric.
\newline

If $X=grad(u)$, equation (\ref{solitricci}) becomes:

\begin{equation*}
Ric(g)+\nabla\nabla u+\mu g=0
\end{equation*}
In this case $g$ is called \textbf{gradient Ricci soliton}.
\newline

For the Yamabe equation one is mostly interested in the compact case. Results of Hamilton \cite{Hamilton} and Ivey \cite{Ivey} give that compact steady
or expanding gradient Ricci solitons are necessarily Einstein. Perelman showed that any compact Ricci soliton is gradient \cite{Perelman}. Therefore, in the compact case, non-trivial Ricci solitons,  that is, Ricci solitons which are not Einstein metrics, are gradient shrinking Ricci solitons.
\newline

Since Ricci solitons are natural generalizations of Einstein metrics, in view of Obata's theorem it seems natural to ask if the Yamabe equation on a closed Ricci soliton only has one solution up to homothecies.
\newline

Note that since Ricci curvature is invariant under rescalings of the metric, for $\mu< 0$ we obtain:
\begin{equation*}
-2Ric(\mu g)=-2Ric(g)=\lied_X g+2\mu g=\lied_{\frac{1}{\mu}X}\mu g+2\mu g.
\end{equation*}
This implies that $\mu g$ is a shrinking Ricci soliton with associated constant $-1$. Then, we may rescale the soliton $g$ so that $\mu$ is $-1$.
\newline

There are few examples of non-trivial Ricci solitons.  The only known compact non-trivial Ricci solitons are rotationally symmetric K\"ahler metrics. For real dimension $4$ the first one was constructed by Koiso \cite{Koiso}, and independently by Cao \cite{Cao1}. It is a non-Einstein shrinking soliton on $\mathbb{CP}^2\#\overline{\mathbb{CP}^2}$ with symmetry $U(2)$ and
positive Ricci curvature. In fact, Cao give a family of homothetically gradient K\"{a}hler-Ricci solitons on $\mathbb{C}^n$ which includes the Hamilton soliton in case $n=1$. The other example in dimension $4$ was found by Wang-Zhu \cite{WangZhu}. They proved the existence of a gradient K\"ahler-Ricci soliton on  $\mathbb{CP}^2\#2\overline{\mathbb{CP}^2}$ with $U(1)\times U(1)$ symmetry. 
\medskip

Compact homogeneous Ricci solitons are Einstein. Thus, next case to study are cohomogeneity one Ricci solitons, studied for instance by Dancer, Hall and Wang \cite{DW}, \cite{DWH}. The only known $4-$dimensional non-trivial metric of cohomogeneity one is the example, mentioned above, constructed by Koiso and Cao on $\mathbb{CP}^2\#\overline{\mathbb{CP}^2}$. In this article we will study the Yamabe equation of this metric for $U(2)-$invariant functions. 
\newline

It is known that on any compact Riemannian manifold $(M, g)$ admitting the action of a compact Lie group $G$, there exists a conformal $G$-invariant metric to $g$ of constant scalar curvature \cite{HebeyVaugon}. Hence, in the case of the Koiso-Cao soliton there exists a $U(2)$-invariant solution to the Yamabe problem. The aim is to prove uniqueness:

\begin{teo}
\label{thm:1}
There exists a unique $U(2)$-invariant solution to the Yamabe equation on  $\mathbb{CP}^2\#\overline{\mathbb{CP}^2}$ with the Koiso-Cao metric.
\end{teo}

This paper is organized as follows. Section 2 is devoted to the description of the Koiso-Cao soliton. We describe the framework for a $U(2)-$invariant metric defined on $S^3\times (\alpha, \beta)$ to be extended to $\mathbb{CP}^2\#\overline{\mathbb{CP}^2}$, as a K\"ahler-Ricci soliton. We derive an ordinary differential equation and its initial conditions, whose solution determines the Koiso-Cao metric. In Section 3 we discuss the curvature of the Koiso-Cao soliton. We prove that it has positive Ricci curvature, and we show that the scalar curvature is a monotone function in the orbit space. We also present the Yamabe equation associated to the Koiso-Cao metric on  $\mathbb{CP}^2\#\overline{\mathbb{CP}^2}$. Finally, Section 4 contains the proof of the main Theorem \ref{thm:1}.
\newline

\textbf{Acknowledgements:} This work is part of my Ph.D. thesis. I want to express my sincere gratitud to my advisors Jimmy Petean and Pablo Su\'arez-Serrato. Their consistent support and positive guidance has been essential to accomplish my academic objectives. I appreciate all different interesting conversations and their generous encourgement. Financial support from The National Council of Science and Technology, CONACyT, through its Fellow Student Program is also acknowledged.

\section{Reviewing the Koiso-Cao Ricci soliton on $\mathbb{CP}^2\#\overline{\mathbb{CP}^2}$}

Koiso and Cao constructed $U(n)-$invariant gradient shrinking K\"ahler-Ricci solitons on twisted projective line bundles over $\mathbb{CP}^{n-1}$ for $n\geq 2$. Cao's construction consists in giving conditions of the K\"ahler potential corresponding to a $U(n)-$invariant K\"ahler metric defined on $\mathbb{C}^n\setminus \{0\}$ to obtain a compact shrinking Ricci soliton. The K\"ahler potential is a smooth function defined on $(-\infty, \infty)$ with certain asymptotic conditions at $-\infty$ and $\infty$. 
\medskip

In this section we will review this construction from a different point of view for $n=2$, which is the Koiso-Cao soliton. This will help us to study the Yamabe equation. For further details see \cite{Koda}. 
\newline

A $U(2)-$invariant metric $(\mathbb{CP}^2\#\overline{\mathbb{CP}^2}, \mathbf{g})$ can be described in the following way. 
The regular orbits of the $U(2)-$action is an open dense subset diffeomorphic to $S^3\times (\alpha, \beta)$, and $U(2)$ acts on $S^3$.  There are two singular orbits diffeomorphic to $S^2$. The invariant metric $g$ on $S^3\times (\alpha, \beta)$ is written as $g=dt^2+g_t$, where $g_t$ is a $U(2)-$invariant metric on $S^3$.

Let $f$ and $h$ be positive smooth functions defined on $(\alpha, \beta)$. For each $t\in(\alpha, \beta)$ the $U(2)-$invariant metric $g_t$ is such that the principal $(S^1, f^2(t)g_o)-$bundle with  projection $\pi: (S^3, g_t)\longrightarrow (S^2, h^2(t) g_o^2)$ is a Riemannian submersion. Here $\pi: S^3\longrightarrow S^2$ is the Hopf fibration, and $g_o$ and $g_o^2$ are the round metrics on $S^1$ and $S^2$, respectively. 
\medskip
%
%
%

The metric $g$ can be extended to a smooth metric $\mathbf{g}$ on $\mathbb{CP}^2\#\overline{\mathbb{CP}^2}$ provided the following asympthotical conditions hold:
\begin{eqnarray}
\label{eqn:extcomp}
f(\alpha)=f(\beta)=0, \quad f'(\alpha)=-f'(\beta)=1, \nonumber \\
h(\alpha)\neq h(\beta)\neq 0, \quad  h'(\alpha)=h'(\beta)=0, \\
f^{2k}(\alpha)=f^{2k}(\beta)= h^{2k+1}(\alpha)=h^{2k+1}(\beta)=0. \nonumber
\end{eqnarray}

If we let $X$, $Y$ and $Z$ be $SU(2)$--left invariant vector fields on $S^3$ given by:
\begin{equation*}
X:\left(\begin{array}{c} v\\ w\end{array}\right)\to \left(\begin{array}{c} iv\\ -iw\end{array}\right), \quad Y:\left(\begin{array}{c} v\\ w\end{array}\right)\to \left(\begin{array}{c} w\\ -v\end{array}\right), \quad Z:\left(\begin{array}{c} v\\ w\end{array}\right)\to \left(\begin{array}{c} iw\\ iv\end{array}\right)
\end{equation*}

\noindent and we call $H=\frac{\partial}{\partial t}$, then
$\displaystyle{E=\left\{H, \frac{X}{f}, \frac{Y}{h}, \frac{Z}{h}\right\}}$ is an orthonormal frame on $(S^3\times (\alpha, \beta), g)$. 
\newline

We have the following commuting relations:
\begin{equation*}\label{eqns:LieB}
[X,Y ]=2Z \quad \quad [Y, Z]=2X \quad \quad [Z, X]=2Y \quad \quad
[H, T]=0
\end{equation*}
for every $T\in E$. Then the Levi-Civita connection induced by $g$ can be computed using Koszul's formula. The following Table contains the computations in terms of the functions $f$ and $h$.
\begin{table}[H]
\centering
\begin{tabular}{l |c |l |l |r}
$\nabla$ & $X$ & $Y$ & $Z$ & $H$\\
\hline

$X$ & $-ff'H$ & $\frac{2h^2-f^2}{h^2}Z$ & $\frac{(f^2-2h^2)}{h^2} Y$ & $\frac{f'}{f} X$ \\

$Y$ & $-\frac{f^2}{h^2}Z$ & $-hh'H$ & $X$ & $\frac{h'}{h}Y$\\

$Z$ & $\frac{f^2}{h^2}Y$ & $-X$ & $-hh'H$ & $\frac{h'}{h}Z$\\

$H$ & $\frac{f'}{f}X$ & $\frac{h'}{h}Y$ & $\frac{h'}{h}Z$ & 0
\end{tabular}
\caption{Levi-Civita connection}
\label{tab:Conex}
\end{table}
\medskip

The almost complex structure $J$ of $\mathbb{CP}^2\#\overline{\mathbb{CP}^2}$ restricted to $S^3\times
(\alpha, \beta)$ is  given by:
\begin{equation*}
J(H)=\frac{X}{f},\quad \quad J(\frac{X}{f})=-H, \quad \quad
J(Y)=Z,\quad \quad J(Z)=-Y.
\end{equation*}

Recall that the Hermitian metric  $g$ is K\"ahler if and only if
\begin{equation*}
\nabla_XJY=J\nabla_X Y
\end{equation*}

\noindent for every $X, Y$ vector fields on $S^3\times(\alpha, \beta)$. From Table (\ref{tab:Conex}) we obtain the computations of $J\nabla_A B$, for every $A, B\in E$:

\begin{table}[H]
\centering
\begin{tabular}{l |c |l |l |r}
$J\nabla$ & $X$ & $Y$ & $Z$ & $H$\\
\hline

$X$ & $-f'X$ & $\frac{f^2-2h^2}{h^2}Y$ & $\frac{(f^2-2h^2)}{h^2} Z$ & $-f'H$ \\

$Y$ & $\frac{f^2}{h^2}Y$ & $-\frac{hh'}{f}X$ & $-fH$ & $\frac{h'}{h}Z$\\

$Z$ & $\frac{f^2}{h^2}Z$ & $fH$ & $-\frac{hh'}{f}X$ & $-\frac{h'}{h}Y$\\

$H$ & $-f'H$ & $\frac{h'}{h}Z$ & $-\frac{h'}{h}Y$ & 0
\end{tabular}
\caption{$J\nabla_A B$}
\label{tab:JConex}

\end{table}

We also compute:
\medskip

\begin{table}[H]
\centering
\begin{tabular}{l ||c |l |l |r}
$\nabla$ & $JX=-fH$ & $JY=Z$ & $JZ=-Y$ & $JH=\frac{X}{f}$\\
\hline

$X$ & $-f'X$ & $\frac{(f^2-2h^2)}{h^2} Y$ & $\frac{(f^2-2h^2)}{h^2} Z$& $-f'H$ \\

$Y$ & $-\frac{fh'}{h}Y$ & $X$ & $hh' H$ & $-\frac{f}{h^2}Z$\\

$Z$ & $-\frac{fh'}{h}Z$ & $-hh'H$ & $X$ & $\frac{f}{h^2}Y$\\

$H$ & $-f'H$ & $\frac{h'}{h}Z$ & $-\frac{h'}{h}Y$ & 0
\end{tabular}
\caption{$\nabla_A J(B)$}
\label{tab:ConexJ}
\end{table}

By comparing Tables (\ref{tab:ConexJ}) and (\ref{tab:JConex}) we deduce  $g$ is a K\"ahler metric if:
\begin{equation}\label{eqn:CondK}
f=-h h'.
\end{equation}

This and conditions (\ref{eqn:extcomp}) imply:
\begin{eqnarray}
\label{eqns:Condfh}
h(\alpha)h''(\alpha)=-h(\beta)h''(\beta)=-1.
\end{eqnarray}

Therefore, for any positive function $h$ defined on $[\alpha, \beta]$ satisfying $h'(\alpha)=h'(\beta)=0$, and (\ref{eqns:Condfh}), $(S^3\times(\alpha, \beta), g)$ is a K\"ahler metric which extends to a K\"ahler metric $\mathbf{g}$ on $\mathbb{CP}^2\#\overline{\mathbb{CP}^2}$.
\newline

We have that $g$ is a $U(2)-$invariant gradient shrinking Ricci soliton if there exists a $U(n)-$invariant function $u$ such that the equation
\begin{equation}
\label{eqn:Solg}
Ric(g)+Hess(u)- g=0
\end{equation}
holds. 
\medskip

Since the K\"ahler metric $g$ is in particular Hermitian, it satisfies $g(JU, JV)=g(U, V)$ for every $U, V$ vector fields. It also holds $Ric(JU, JV)=Ric(U, V)$. This implies that, if in addition $g$ is a gradient Ricci soliton, $Hess(u)(JU, JV)=Hess(u)(U, V)$. 
\newline

Let $\psi$ be a smooth function on $S^3\times (\alpha,\beta)$ invariant by the action of $U(2)$,  and simply write $\psi: (\alpha, \beta)\to \mathbb{R}$. We compute the Hessian of $\psi$ using Table (\ref{tab:Conex}).
\newline

In the orthonormal basis $\displaystyle{\left\{H, \frac{X}{f}, \frac{Y}{h}, \frac{Z}{h}\right\}}$,
the Hessian of $\psi$ diagonal and
\begin{eqnarray}
\label{eqns:Hess}
Hess(\psi)\left(\frac{X}{f}, \frac{X}{f}\right)&=&\frac{f' \psi'}{f} \nonumber \\
Hess(\psi)\left(\frac{Y}{h}, \frac{Y}{h}\right)&=&\frac{h' \psi'}{h}\nonumber \\
Hess(\psi)\left(\frac{Z}{h}, \frac{Z}{h}\right)&=&\frac{h' \psi'}{h}\\
Hess(\psi)(H, H)&=&\psi''.\nonumber
\end{eqnarray}

If $u$ is a potential function  determining a gradient shrinking Ricci soliton $g$, the Hessian of $u$ is $J-$invariant. This happens if and only if
\begin{eqnarray*}
Hess(u)(H, H) &=&Hess(u)(JH, JH)=Hess(u)\left(\frac{X}{f}, \frac{X}{f}\right)\\
Hess(u)\left(\frac{Y}{h}, \frac{Y}{h}\right)&=&Hess(u)\left(J\frac{Y}{h},J\frac{Y}{h}\right)=Hess(u)\left(J\frac{Z}{h},J\frac{Z}{h}\right),	
\end{eqnarray*}

\noindent that is, if and only if
\begin{equation*}
\label{eqn:Hessian-invariant}
u''=\frac{f'u'}{f}.
\end{equation*}

Solving for $u'$ we obtain:
\begin{equation*}
u'=cf
\end{equation*}

\noindent for some constant $c$. Note that $c\neq 0$, if $g$ is a non-trivial Ricci soliton. Substituting (\ref{eqn:CondK}) we obtain:
\begin{equation*}
u'=-chh'=-\frac{c}{2}\frac{d}{dt}h^2.
\end{equation*}

Then
\begin{equation*}
u=\frac{-ch^2}{2}+d
\end{equation*}

\noindent for some constants $c,d$, with $c\neq 0$. Without loss of generality, we may set $d=0$, and so:
\begin{equation}
\label{fun:Ricci}
u=\frac{-ch^2}{2}
\end{equation}

On the other hand, the Ricci curvature is given by: 
\begin{equation}
\label{eqn:RicciCurv1}
Ric(H, H)=-\frac{f''}{f}-2\frac{h''}{h}
\end{equation}
\begin{equation}
\label{eqn:RicciCurv2}
Ric\left(\frac{X}{f}, \frac{X}{f}\right)=-\frac{f''}{f}-2\frac{f'h'}{fh}+2\frac{f^2}{h^4}
\end{equation}
\begin{equation}
\label{eqn:RicciCurv3}
Ric\left(\frac{Y}{h}, \frac{Y}{h}\right)=Ric\left(\frac{Z}{h}, \frac{Z}{h}\right)=-\frac{h''}{h}-\frac{f'h'}{fh}-\frac{h'^2}{h^2}+\frac{4}{h^2}-2\frac{f^2}{h^4}.
\end{equation}

Since also Ricci curvature has to be $J$--invariant, $Ric(H, H)=Ric(\frac{X}{f}, \frac{X}{f})$. Hence:
\begin{equation*}
\frac{h''}{h}-\frac{f^2}{h^4}-\frac{f'h'}{fh}=0.
\end{equation*}
\medskip

Now we write the soliton equation using the previous calculations for each direction.
\medskip

In the direction of $H$, by (\ref{eqn:CondK}) we have:
\begin{eqnarray*}
\label{eqns:SolH}
0&=&Ric(H, H)+Hess(u)(H, H)- g(H, H)\\
&=&-\frac{f''}{f}-2\frac{h''}{h}+u''-1\\
&=&-\frac{f''}{f}-2\frac{h''}{h}-c h'^2-c h h''-1\\
&=&-\frac{5h''}{h}-\frac{h'''}{h'}-1-c(hh''+h'^2).
\end{eqnarray*}

Then we obtain the equation:
\begin{equation}\label{eqn:SolH}
\frac{h'''}{h'}+\frac{5h''}{h}+1+c(hh''+h'^2)=0
\end{equation}

Proceeding similarly in the direction of $X$,
\begin{eqnarray*}
0&=&Ric(X, X)+Hess(u)(X, X)- g(X, X)\\
&=&f^2\left(-\frac{f''}{f}-2\frac{f'h'}{fh}+2\frac{f^2}{h^4}\right)+cf^2f'-f^2\\
&=&-5hh'^2h''-h^2h'h'''-h^2h'^2-ch^2h'^2(hh''+h'^2),
\end{eqnarray*}
and then
\begin{equation}\label{eqn:SolX}
h^2h'h'''+5hh'^2h''+h^2h'^2(1+c(hh''+h'^2))=0,
\end{equation}
which is equivalent to (\ref{eqn:SolH}).
\newline

In the direction of $Y$, 

\begin{eqnarray*}
0&=&Ric(Y, Y)+Hess(u)(Y, Y)- g(Y, Y)\\
&=&h^2\left(-\frac{h''}{h}-\frac{f'h'}{fh}-\frac{h'^2}{h^2}+\frac{4}{h^2}-2\frac{f^2}{h^4}\right)+hh'u'-h^2\\
&=&h^2\left(-\frac{h''}{h}-\frac{f'h'}{fh}-\frac{h'^2}{h^2}+\frac{4}{h^2}-2\frac{f^2}{h^4}\right)-ch^2h'^2-h^2\\
&=&-2hh''-4h'^2+4-ch^2h'^2-h^2.
\end{eqnarray*}
We then obtain:
\begin{equation}\label{eqn:SolY}
2hh''+4h'^2-4+h^2(1+ch'^2)=0.
\end{equation}

If we differentiate equation (\ref{eqn:SolY}) with respect to $t$, we obtain:
\begin{eqnarray*}
0&=&2(hh'''+h'h'')+8h'h''+h^2(2ch'h'')+2hh'(1+ch'^2)\\
&=&2hh'''+10h'h''+2ch^2h'h''+2hh'+2chh'^3.
\end{eqnarray*}

Then
\begin{equation*}
hh'''+5h'h''+hh'(1+c(hh''+h'^2))=0
\end{equation*}
It follows that if $h$ solves (\ref{eqn:SolY}), it also solves (\ref{eqn:SolH}) and (\ref{eqn:SolX}).
\newline

We have:

\begin{pro}
\label{prop:SolK}
The $U(2)-$invariant K\"ahler metric $g=dt^2+g_t$ defined on $S^3\times(\alpha, \beta)$ by the function $h$, extends to a gradient shrinking K\"ahler-Ricci soliton $\mathbf{g}$ on $\mathbb{CP}^2\#\overline{\mathbb{CP}^2}$ if $h$ is a smooth positive function which, for a constant $c\in \mathbb{R}$, solves the equation:
\begin{equation}
\label{eqn:SolR}
2hh''+4h'^2-4+h^2(1+ch'^2)=0
\end{equation} 

\noindent with $h'(\alpha)=h'(\beta)=0$, and $h(\alpha)h''(\alpha)=-h(\beta)h''(\beta)=-1$. It follows that $h(\alpha)=\sqrt{6}$ and $h(\beta)=\sqrt{2}$.
\end{pro}

Without loss of generality we may choose $\alpha=0$.
\medskip 

\begin{pro}
\label{prop:Uniquec}
There exists a unique constant $c_o$ such that for the solution of (\ref{eqn:SolR}) satisfying $h_{c_o}(0)=\sqrt{6}$, $h_{c_o}'(0)=0$ there exists $\beta>0$ such that $h_{c_o}'<0$ on $(0, \beta)$, $h_{c_o}'(\beta)=0$ and $h_{c_o}(\beta)=\sqrt{2}$. If there exists $\beta>0$ for which $h_c'<0$ on $(0, \beta)$, $h_c'(\beta)=0$ and $h_c(\beta)>\sqrt{2}$ (or $h_c(\beta)<\sqrt{2}$) then $c_o<c$ (resp. $c<c_o$). The constant $c_o$ can be computed numerically, its aproximation is $c_o=-0.5276195198969626$.  
\end{pro}

\begin{proof}
Given a constant  $c\neq 0$ and $s>0$ we have a solution $h_{c,s}$ to (\ref{eqn:SolR}) satisfying $h_{c,s}(0)=s$, ${h_{c,s}}'(0)=0$. Note that if $h_{c,s}(t)$ is solution to (\ref{eqn:SolR}), with ${h_{c,s}}'(0)=0$, then $h_{c,s}(-t)$ is also a solution. Or more generally, a solution to (\ref{eqn:SolR}) will be symmetric around any critical point.
\newline

Let $v=h'$, then  from equation (\ref{eqn:SolR}) we obtain the following non-linear autonomous system:
\begin{eqnarray}
\label{eqns:system}
h'&=&v \nonumber\\
v'&=&\frac{2}{h}-2\frac{v^2}{h}-\frac{h}{2}(1+cv^2) 
\end{eqnarray}

We may depict its corresponding phase portrait in coordinates $(h, v)$ in the right-half plane $\mathbb{R}^{+}\times \mathbb{R}\subset \mathbb{R}\times \mathbb{R}$. Observe that $(2,0 )$ is the only critical point of the system.
\newline

Note that if $h_{c,s}$ has a maximum and a minimum, then it is periodic and the corresponding integral curve of (\ref{eqns:system}) is closed. Note also that $h_{c,s}$ has a maximum at $0$ if $s>2$, and if it has its first minimum at $\beta_{c,s}>0$, then $0<h_{c,s}(\beta_{c, s_1})<2$. In this case ${h_{c,s}}'<0$ on $(0, \beta_{c,s})$.  
\newline

Let $d$ and $c$, with $c<d$,  and $s_1, s_2>2$, such that the solutions $h_{c,s_1}$ and $h_{d,s_2}$ to (\ref{eqn:SolR}), with initial conditions ${h_{c,s_1}}'(0)={h_{d,s_2}}'(0)=0$, have first minimums at $\beta_{c,s_1}$ and $\beta_{d,s_2}$, respectively. Denote by $\gamma_{c}^{s_1}$ and $\gamma_{d}^{s_2}$ the corresponding integral curves, they are closed curves. 
\newline

If $\gamma_{c}^{s_1}$ and $\gamma_{d}^{s_2}$ intersect in the lower-half plane, there exist $t_1, t_2>0$ such that $h_{c,s_1}(t_2)=h_{d,s_2}(t_1)>0$, and ${h_{c,s_1}}'(t_2)={h_{d,s_2}}'(t_1)<0$. Then we have two equations:
\begin{equation}
\label{eqn:1}
2h_{c, s_1}(t_2)h''_{c, s_1}(t_2)+4h'^2_{c, s_1}(t_2)-4+h_{c, s_1}^2(t_2)(1+ch'^2_{c, s_1}(t_2))=0
\end{equation}
\begin{equation}
\label{eqn:2}
2h_{d, s_2}(t_1)h''_{d, s_2}(t_1)+4h'^2_{d, s_2}(t_1)-4+h_{d, s_2}^2(t_1)(1+dh'^2_{d, s_2}(t_1))=0
\end{equation}

Also we obtain:
\begin{equation}
\label{eqn:3}
2h_{d, s_2}(t_1)h_{c, s_1}''(t_2)+4h'^2_{d, s_2}(t_1)-4+h_{d, s_2}^2(t_1)(1+ch'^2_{d, s_2}(t_1))=0
\end{equation}

If we substract (\ref{eqn:3}) from (\ref{eqn:2}):
\begin{equation*}
2h_{d, s_2}(t_1)(h''_{d, s_2}(t_1)-h''_{c, s_1}(t_2))+h_{d, s_2}^2(t_1)h'_{d, s_2}(t_2)(d-c)=0
\end{equation*}

Therefore, since $h_{d, s_2}(t_1)>0$ and $h'_{d, s_2}(t_2)<0$, we obtain $h''_{d, s_2}(t_1)>h''_{c, s_1}(t_2)$. It follows that, since 
$$\frac{d}{dt}\gamma_{c}^{s_1}=(h_{c, s_1}', h_{c, s_1}'')=(h_{c, s_1}',\frac{2}{h_{c, s_1}}-2\frac{{h'_{c, s_1}}^2}{h_{c, s_1}}-\frac{h_{c, s_1}}{2}(1+c{h'_{c, s_1}}^2) ),$$
and
$$\frac{d}{dt}\gamma_{d}^{s_2}=(h_{d, s_2}', h_{d, s_2}'')=(h_{d, s_2}',\frac{2}{h_{d, s_2}}-2\frac{{h_{d, s_2}'}^2}{h_{d, s_2}}-\frac{h_{d, s_2}}{2}(1+d{h'_{d, s_2}}^2) ),$$

if $\gamma_{c}^{s_1}$ and $\gamma_{d}^{s_2}$ intersect in the lower-half plane, after the crossing point $\gamma_{c}^{s_1}$ goes above $\gamma_{d}^{s_2}$. It follows that if $h_{c, s_1}$ has a minimum, so does $h_{d, s_2}$, and $h_{d, s_2}(\beta_{d, s_2})>h_{c, s_1}(\beta_{c, s_1})$. 
\medskip
\begin{figure}
    \centering
    \includegraphics[width=.9\textwidth]{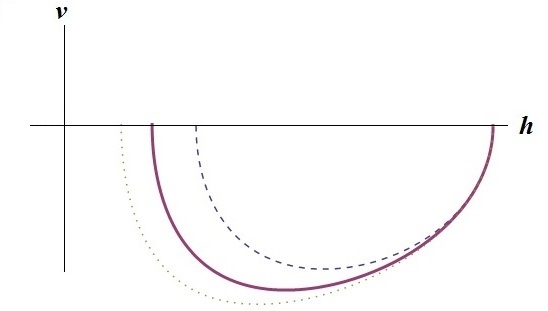}
    \caption{\footnotesize A diagram depicting the behavior of three solutions of the system (\ref{eqns:system}) with same initial conditions. For $c<c_o<d$, the solid line corresponds to $h_{c_o,\sqrt{6}}$, the dashed one to $h_{d,\sqrt{6}}$ and the dotted one to $h_{c, \sqrt{6}}$.}
    \label{fig:phaseportrait}
\end{figure}
\medskip

Now, take $s_1=s_2=\sqrt{6}$, $c<d$ as before, and such that $h_c=h_{c, \sqrt{6}}$ and $h_d=h_{d, \sqrt{6}}$, both have positive minimums $m_c$ and $m_d$, respectively. We have their corresponding integral curves $\gamma_c$ and $\gamma_d$.

Define the function $min: (c, d)\to(0, 2)$ which assigns to each $a\in (c, d)$ the first minimum value of the solution $h_a^{\sqrt{6}}$. 
\newline

\noindent{\textbf{Claim:} The function $min$ is well defined, continuous and increasing.}
\medskip

If $\gamma_c$ goes under $\gamma_d$, it has to intersect some $\gamma_d^s$, with $2<s<\sqrt{6}$. But then $\gamma_c$ also intersects every integral curve $\gamma_d^t$, with $t<s$. Take $\gamma_d^{\frac{s+\sqrt{6}}{2}}$, for instance. Since $d>c$, by previous discussion, $\gamma_c$ stays below $\gamma_d^{\frac{s+\sqrt{6}}{2}}$ and cannot intersect $\gamma_d^s$. Then, in particular $m_c<m_d$. Same argument applies to each $a\in(c, d)$, the solution $h_{a, \sqrt{6}}$ has a minimum. Then the associated integral curve is closed, since $\sqrt{6}>2$. We have that $m_d<min(a)<m_c$.
\medskip

Therefore, it can only exist one $c_o\in(c, d)$ such that the solution $h_{c_o,\sqrt{6}}$ to (\ref{eqn:SolR}) with $h_{c_o, \sqrt{6}}'(0)=0$ has as minimum value $\sqrt{2}$. 
\newline

Finally, the value of $c_o$ can be computed numerically by solving the equation (\ref{eqn:SolR}) for different parameters $c$, and initial conditions $h(0)=\sqrt{6}$ and $h'(0)=0$. By the claim, if $c<d$ are two values for which the solutions $h_c$ and $h_d$ each have minimums $m_c, m_d$, and $m_d<\sqrt{2}<m_c$ then we know that $c<c_o<d$. Then $c_o$ can be approximated by interpolation.	
\newline

Cao  also obtained the value of $c_o$ as root of the function:
\begin{equation*}
k(x)=e^{2x} (2 - 4 x + 3 x^2) - 2 + x^2.
\end{equation*}

See Lemma 4.1. in \cite{Cao1}.
\end{proof}

\section{Scalar curvature and the Yamabe equation}

In this section we will write the Yamabe equation on $\mathbb{CP}^2\#\overline{\mathbb{CP}^2}$ for a $U(2)-$invariant function. Recall that $\mathbf{g}$ has positive Ricci curvature \cite{Cao1}. We will use this to see that $S_\mathbf{g}$ is a decreasing function of $t$ on $S^3\times(\alpha, \beta)$. For the sake of completeness we give a proof of the positivity of Ricci curvature using our description of $\mathbf{g}$.

\begin{pro}
\label{prop:PosRic}
The Koiso- Cao soliton $\mathbf{g}$ has positive Ricci curvature.
\end{pro}

\begin{proof}

For a gradient shrinking soliton we have the equation (\ref{eqn:Solg}). Note that since $Hess(u)$ is diagonal in the orthonormal basis $E$, then so is the Ricci curvature of $\mathbf{g}$. Recall also that $\mathbf{g}$ is K\"ahler, then
\begin{eqnarray*}
Ric\left(\frac{X}{f}, \frac{X}{f}\right)=Ric(H, H)\\
Ric\left(\frac{Z}{h}, \frac{Z}{h}\right)=Ric\left(\frac{Y}{h}, \frac{Y}{h}\right).
\end{eqnarray*}

Then we only have to show that the functions $Ric(H, H)$ and $Ric(\frac{Y}{h}, \frac{Y}{h})$ are positive. Using the Hessian (\ref{eqns:Hess}) and the expression of $u$ (\ref{fun:Ricci}) we obtain
\begin{eqnarray*}
Ric(H, H)&=&g(H, H)-Hess(u)(H, H)\\
&=&1-u''\\
&=&1+c(hh''+h'^2),
\end{eqnarray*}
and 
\begin{eqnarray*}
Ric\left(\frac{Y}{h}, \frac{Y}{h}\right)&=&g\left(\frac{Y}{h}, \frac{Y}{h}\right)-Hess(u)\left(\frac{Y}{h}, \frac{Y}{h}\right)\\
&=&1-\frac{h'u'}{h}\\
&=&1+ch'^2.
\end{eqnarray*}

First we will lead with the Ricci curvature in the direction of $Y$. From equation (\ref{eqn:SolR}) we rewrite
\begin{equation*}
(hh''+h'^2)+h'^2-2+\frac{h^2}{2}(1+ch'^2)=0
\end{equation*}
which is the same as:
\begin{equation}
\label{eqn:f'}
f'=h'^2-2+\frac{h^2}{2}(1+ch'^2)
\end{equation}

Let $A(x)=x-2+\frac{b^2}{2}(1+cx)$. If we differentiate with respect to $x$,
\begin{equation*}
\frac{d}{d x}A=1+c\frac{b^2}{2}
\end{equation*}

The critical points of $A$ occurs when $\frac{b^2}{2}=-1/c$. 
\newline

\noindent{\textbf{Claim:} $1+c h'^2>0$. }
\medskip

We have a critical point of the function $(1+c h'^2)$ if $h'=0$ or $h''=0$. Observe that that $h'=0$ and $h''=0$ can not happen simultaneously, else the value of the critical points would be $2$. This is not possible since critical values have to be $\sqrt{2}$ and $\sqrt{6}$. If $h'=0$, the claim follows and then $Ric(\frac{Y}{h}, \frac{Y}{h})>0$ at $0$ and $\beta$. For those points where $h''=0$,  from equation (\ref{eqn:SolR}):
\begin{equation}
h'^2=\frac{4-h^2}{4+ch^2}
\end{equation}
We will considerate two cases. 

\begin{enumerate}
\item [i)] If $3>\frac{h^2}{2}\geq -1/c$. Then, $4-h^2<4+2/c$. From the value of $c=-0.5276195198969626$ we have $4+2/c<1$. Also, $4+6c<4+ch^2$. If follows that
\medskip

\begin{equation*}
h'^2< \frac{4+2/c}{4+6c}<\frac{1}{4+6c} < -1/c,
\end{equation*}
and then $1+ch'^2>0$. 

\item [ii)] If $1<\frac{h^2}{2}<-1/c$, we have $4-h^2<2$ and $2<4+ch^2$. Then:
\begin{equation*}
h'^2=\frac{4-h^2}{4+ch^2}<1
\end{equation*}
The claim is proved, and then $Ric(\frac{Y}{h}, \frac{Y}{h})>0$.
\end{enumerate}
\medskip

On the other hand, for the Ricci curvature in the direction of $H$, note that,  since $-1<c<1$ and $h(0)h''(0)=-h(\beta)h''(\beta)=1$, $Ric(H, H)$ is positive at $0$ and $\beta$. Then it remains to prove that $1+c(hh''+h'^2)>0$ for every $t\in(0, \beta)$. 

Since $c<0$, we have to prove $f'=(hh''+h'^2)>1/c\approx -1.8$. It is enough to see $f'>-1$. Hence we will show the bound for $f'$ in the two previous cases. 
\medskip

Using equation (\ref{eqn:f'}), the fact that $Ric(\frac{Y}{h}, \frac{Y}{h})>0$ and the value of $c$:
\begin{enumerate}

\item [i)] If $\frac{h^2}{2}\geq -1/c$ we have, 
\begin{eqnarray*}
f'&=&h'^2-2+\frac{h^2}{2}(1+ch'^2)\\
&\geq& h'^2-2-\frac{1}{c}(1+ch'^2)\\
&=&-2-1/c>-1
\end{eqnarray*}

\item [ii)] If $1<\frac{h^2}{2}<-1/c$, first note that $(1+\frac{ch^2}{2})>0$, and since $1<\frac{h^2}{2}$:
\begin{eqnarray*}
f'&=&h'^2(1+\frac{ch^2}{2})+\frac{h^2}{2}-2\\
&>&\frac{h^2}{2}-2>-1
\end{eqnarray*}
\end{enumerate}

Therefore, $f'>-1$, and so $Ric(H, H)>0$.
\end{proof}
\newline

Let $\psi$ be a smooth function on $S^3\times (0,\beta)$ invariant by the action of $U(2)$, then by (\ref{eqns:Hess}) we get
\begin{eqnarray}
\label{Laplacian}
\triangle_g\psi&=&-\frac{f' \psi'}{f}-2\frac{h' \psi'}{h}-\psi''\nonumber\\
&=&-\psi'\left(\frac{hh''+h'^2}{hh'}+2\frac{h'}{h}\right)-\psi''\nonumber\\
&=&-\psi'\left(\frac{h''}{h'}+3\frac{h'}{h}\right)-\psi''.
\end{eqnarray}

Then for $u=-\frac{ch^2}{2}$ we have:
\begin{eqnarray*}
\triangle u&=&\frac{d}{dt}\left(\frac{ch^2}{2}\right)\left(\frac{h''}{h'}+3\frac{h'}{h}\right)+\frac{d^2}{dt^2}\left(\frac{ch^2}{2}\right)\\
&=&chh''+3ch'^2+chh''+ch'^2\\
&=&4ch'^2+2 chh''
\end{eqnarray*}

Now, taking the trace of the soliton equation (\ref{eqn:Solg}) we obtain
\begin{equation*}
S-\triangle u=4
\end{equation*}
and therefore the scalar curvature is:
\begin{equation}
\label{eqn:Scalar}
S=4ch'^2+2ch h''+4.
\end{equation}

From Proposition \ref{prop:SolK} and Proposition \ref{prop:Uniquec} we obtain:
\begin{eqnarray}
\label{value:S}
S(0)&=&4-2c>0\nonumber\\
S(\beta)&=&4+2c>0
\end{eqnarray}

\begin{pro}
\label{prop:PosScal}
The scalar curvature of the Koiso-Cao soliton is decreasing as a function of $t\in[0, \beta]$.
\end{pro}

\begin{proof}
Take the derivative with respect to $t$ of (\ref{eqn:Scalar}), then:
\begin{equation*}
S'=8ch'h''+2c(hh'''+h'h'')=10ch'h''+2chh'''
\end{equation*}

But,
\begin{eqnarray*}
Ric(H, H)&=&-\frac{f''}{f}-2\frac{h''}{h}\\
&=&-\frac{hh'''+5h'h''}{hh'}.
\end{eqnarray*}

Since $Ric(H, H)>0$, we have that $hh'''+5h'h''>0$. Therefore, since $c<0$ , $S'<0$ on $(0, \beta)$. 
\end{proof}
\newline

Let $\varphi$ be a smooth function invariant under the action of $U(2)$ on $\mathbb{CP}^2\#\overline{\mathbb{CP}^2}$. Then we identify $\varphi$ with a function $\phi:[0, \beta]\to \mathbb{R}$ such that $\phi'(0)=\phi'(\beta)=0$. We use previous calculation of the Laplacian (\ref{Laplacian}) to obtain the Yamabe equation of the Koiso-Cao soliton:
\begin{equation}\label{eqn:Yamabe}
6\phi'\left(\frac{h''}{h'}+3\frac{h'}{h}\right)+6\phi''+S_{\mathbf{g}}\phi=\lambda \phi^3
\end{equation}

Since Koiso-Cao soliton has positive scalar curvature, then $\lambda$ has to be positive. We fix $\lambda=1$. 

\section{Proof of Theorem \ref{thm:1}}

Let $\phi$ be a positive $U(2)-$invariant function on $S^3\times (0, \beta)$ with the K\"ahler metric $g$. In order to $\phi^2\cdot \mathbf{g}$ be a Riemannian metric on $\mathbb{CP}^2\#\overline{\mathbb{CP}^2}$ , $\phi$ must satisfy $\phi'(0)=\phi'(\beta)=0$.  Additionally, it has constant scalar curvature equals $1$ if $\phi$ satisfies the Yamabe equation (\ref{eqn:Yamabe}):
\begin{equation*}
6\phi'\left(\frac{h''}{h'}+3\frac{h'}{h}\right)+6\phi''+S_{\mathbf{g}}\phi= \phi^3
\end{equation*}

\noindent{Since:}
\begin{eqnarray*}
\lim_{t\to 0}\phi'\frac{h''}{h'}&=&\lim_{t\to 0}\left(\frac{\phi'h'''}{h''}+\phi''\right)\\
&=&\phi''(0)
\end{eqnarray*}

\noindent{then we have:}
\begin{equation*}
12\phi''(0)=\phi(0)(\phi^2(0)-S(0))
\end{equation*}

\noindent{For $t\in(0, \beta)$ with $\phi'(t)=0$,} 
\begin{equation*}
6\phi''(t)=\phi(t)(\phi^2(t)-S(t)).
\end{equation*}

\textbf{Claim:}  $\phi^2(0)\leq S(0)=4-2c$ (by (\ref{value:S})).
\medskip

If  $\phi^2(0) >S(0)$ then $0$ is a local minimum. $\phi$ increases and $S$ decreases. After $0$, at the next critical point of $\phi$ we would again have that $\phi^2>S$ and it would be a minimum.
\medskip 

In particular, same argument says that there is no local minimums on $[0, \beta)$. Therefore, $0$ is a local maximum and $\beta$ a local minimum, and there is no more critical points of $\phi$ on $(0,  \beta)$. Then $\phi$ is decreasing. 
\newline

\noindent{We have:}
\begin{equation*}
\sqrt{S(\beta)}\leq\phi(\beta)\leq\phi(0)\leq\sqrt{S(0)}
\end{equation*}

Remember that one $U(2)$-invariant solution to Yamabe equation is guaranteed (see Hebey-Vaugon \cite{HebeyVaugon}. It remains to prove uniqueness. If $s\in(\sqrt{S(\beta)}, \sqrt{S(0)}]$, let $\phi_s$ be a solution to the Yamabe equation (\ref{eqn:Yamabe}) such that $\phi_s(0)=s$. Let $s_1$, $s_2\in (\sqrt{S(\beta)}, \sqrt{S(0)}]$, $s_1\tl s_2$. We have $0$  is a local maximum for both solutions $\phi_{s_{1}}$ and $\phi_{s_{2}}$. Define $F(t):=\phi_{s_2}-\phi_{s_1}$.
\newline
 
\textbf{Claim:}\ \ $F$ has a local minimum at $0$.
\medskip

Note that $F(0)>0$. Since $0$ is a critical point of $\phi_{s_1}$ and $\phi_{s_2}$, so it is for $F$. Since $\phi_{s_1}$ and $\phi_{s_2}$ satisfy (\ref{eqn:Yamabe}), it follows that:
\begin{equation*}
F''(0)=\frac{1}{12}\left(\phi_{s_2}^3(0)-\phi_{s_2}(0)S(0)- (\phi_{s_1}^3(0)-\phi_{s_1}(0)S(0)\right)
\end{equation*}

Consider the function $v(x)=x^3-S(0)x$. Then:
\begin{equation*}
v'(x)=3x^2-S(0)
\end{equation*}

Note that  $x_o:=\sqrt{\frac{S(0)}{3}}$ is the only critical point, and $v$ is increasing on $(\sqrt{\frac{S(0)}{3}}, \sqrt{S(0)} )$.  
\medskip

We have, by (\ref{value:S}) and the value of $c$,
\begin{equation*}
S(0)=4-2c=5.0552\ \ \ \text{and}\ \ \ S(\beta)=4+2c=2.9447
\end{equation*}

then
\begin{equation*}
\sqrt{S(0)}=2.2483,\ \ \ \sqrt{S(\beta)}=1.716,\ \ \ \text{and}\ \ \ \sqrt{\frac{S(0)}{3}}=1.2981.
\end{equation*}

Hence, since
\begin{equation*}
\sqrt{\frac{S(0)}{3}}\tl\sqrt{S(\beta)}\tl\phi_{s_1}(0)\tl\phi_{s_2}(0)\leq\sqrt{S(0)},
\end{equation*}
it follows that $F''(0)\tgt 0$, and so $0$ is a local minimum of $F$.
\medskip

Hence, $F$ is positive and increasing on $(0, \varepsilon)$, for $\varepsilon\tgt 0$ enough small. Assume that there exists $t_o\in (\varepsilon, \beta]$ such that $F'(t_o)=0$, and take the first critical point. The discussion above applies for $t_o$, that is, we have:
\begin{equation*}
F''(t_o)=\frac{1}{6}\left(\phi_{s_2}^3(t_o)-\phi_{s_2}(t_o)S(t_o)- (\phi_{s_1}^3(t_o)-\phi_{s_1}(t_o)S(t_o)\right)
\end{equation*}

Similarly as before, consider the function $g(x)=x^3-S(t_o)x$. Then $x_1:=\sqrt{\frac{S(t_o)}{3}}$ is its only critical point, which is a minimum.  From the choosing of $t_o$ it follows :
\begin{equation*}
\sqrt{\frac{S(t_o)}{3}}\tl S(\beta)\tl\phi_{s_1}(t_o)\tl\phi_{s_2}(t_o)
\end{equation*}

Then $t_o$ is a local minimum of $F$. But this implies there are another $t_1\in(0, t_o)$ satisfying  $F'(t_1)=0$, which is a contradiction, since $t_o$ was the first point.
\newline

We have proven that $F'\tgt 0$ on $(0,\beta]$ and in particular $F'(\beta)=\phi_{s_2}(\beta)-\phi_{s_1}(\beta)$>0.  Then it cannot happen that $\phi_{s_1}'(\beta)=\phi_{s_2}'(\beta)=0$. This proves uniqueness.

\medskip

CIMAT, GUANAJUATO, GTO., M\'exico

\textit{E-mail address:} jonatan@cimat.mx

\end{document}